\documentclass{amsart}

\usepackage[foot]{amsaddr}
\usepackage[utf8]{inputenc}
\usepackage[T1]{fontenc}
\usepackage[english]{babel}
\usepackage{amssymb,latexsym}
\usepackage{amsmath}
\usepackage{amsthm}
\usepackage{amssymb}
\usepackage{mathtools}
\usepackage{graphicx}
\usepackage{color}
\usepackage[shortlabels]{enumitem}
\usepackage[colorlinks=true,citecolor={blue},draft=false, urlcolor={black}]{hyperref}
\usepackage{cleveref}

\usepackage{todonotes}
\usepackage[shortlabels]{enumitem}

\def\R{\mathbb R}
\def\Q{\mathbb Q}
\def\N{\mathbb N}
\def\A{\mathcal A}

\def\L{\mathcal L}
\def\uu{\mathbf u}
\def\vv{\mathbf v}
\def\ss{\mathbf s}
\def\ww{\mathbf w}

\def\dd{\mathbf d}

\def\R{\mathbb R}
\def\Q{\mathbb Q}
\def\N{\mathbb N}
\def\A{\mathcal A}

\newtheorem{thm}{Theorem}
\newtheorem{theorem}[thm]{Theorem}
\newtheorem{coro}[thm]{Corollary}

\newtheorem{lem}[thm]{Lemma}
\newtheorem{lmx
m}[thm]{Lemma}
\newtheorem{claim}[thm]{Claim}
\newtheorem{obs}[thm]{Observation}

\newtheorem{prop}[thm]{Proposition}
\newtheorem{defi}[thm]{Definition}

\crefname{thm}{theorem}{theorems}
\crefname{theorem}{theorem}{theorems}
\crefname{coro}{corollary}{corollaries}
\crefname{example}{example}{examples}
\crefname{lem}{lemma}{lemmas}
\crefname{lmm}{lemma}{lemmas}
\crefname{claim}{claim}{claims}
\crefname{obs}{observation}{observations}
\crefname{proposition}{proposition}{propositions}
\crefname{prop}{proposition}{propositions}
\crefname{defi}{definition}{definitions}

\newtheorem{remark}[thm]{Remark}

\newtheorem{example}[thm]{Example}

\crefname{example}{example}{examples}

\begin{document}

\title{On Sturmian substitutions closed under derivation}

\author[Edita Pelantov\'a]{Edita \textsc{Pelantov\'a}}
\author[\v St\v ep\'an Starosta]{\v St\v ep\' an  \textsc{Starosta}}

\address[Edita Pelantov\'a]{
Department of Mathematics, FNSPE, Czech Technical University in Prague\\
Trojanova 13, 120 00 Praha 2, Czech Republic}

\address[\v St\v ep\'an Starosta]{
Department of Applied Mathematics, FIT, Czech Technical University in Prague\\
Thákurova 9, 160 00 Praha 6, Czech Republic}


\begin{abstract}
Occurrences of a factor $w$ in an infinite uniformly recurrent sequence ${\bf u}$ can be encoded by an infinite sequence over a finite alphabet.
This sequence is usually denoted ${\bf d_{\bf u}}(w)$ and called the derived sequence to $w$ in ${\bf u}$.
If $w$ is a prefix of  a fixed point ${\bf u}$  of a primitive substitution  $\varphi$,  then by Durand's result from 1998, the derived sequence
${\bf d_{\bf u}}(w)$ is fixed by a  primitive substitution $\psi$ as well.
For a non-prefix factor $w$, the derived sequence  ${\bf d_{\bf u}}(w)$ is fixed by a substitution only exceptionally.
To   study this phenomenon we introduce a new notion:   A finite set  $M $  of substitutions  is said to be closed  under derivation if the derived sequence ${\bf d_{\bf u}}(w)$  to any factor $w$ of any fixed point ${\bf u}$ of  $\varphi \in M$  is fixed by a morphism $\psi \in M$.
In our article we characterize the Sturmian substitutions which belong to a  set $M$ closed under derivation.
The characterization uses either the slope and the intercept of its fixed point or its S-adic representation.
\end{abstract}


\maketitle


\noindent \textit{2000MSC:} 68R15



\section{Introduction}

In combinatorics of words, the notion of return words to a factor of an infinite word is an analogue to first return map in dynamical systems.
Given an infinite word $\uu = u_0 u_1 u_2 \dots $ with $u_i$ being an element of a finite alphabet, we say that $u_i u_{i+1} \cdots u_{j-1}$ is a return word to a factor $w$ if for each $k$ satisfying $i \leq k \leq j$, the factor $w$ is a prefix of the infinite word $u_{k}u_{k+1}u_{k+2} \dots$ only for $k = i$ and $k = j$.
We study infinite words $\uu$  such that every factor  of $\uu$  occurs infinitely many times and has a finite number of return words.
These words are called uniformly recurrent.
This class of words includes purely periodic words.
Obviously, any  factor  $w$  of a purely periodic word  $\uu$ which is longer than the period   has just one return word.
On the other hand, if a uniformly recurrent word $\uu$ has a factor having only one return word, then $\uu$ is purely periodic.

In this article, we focus on uniformly recurrent words which have exactly 2 return words to each factor.
As shown by Vuillon in \cite{Vu}, such words are exactly the infinite Sturmian words, i.e., aperiodic words having the least factor complexity possible.

If a factor $w$ of a uniformly recurrent word $\uu$ has $k$ return words, then the order of their occurrences in $\uu$ can be coded by an infinite word over a $k$-letter alphabet.
This word is denoted by $\dd_\uu(w)$ and called the derived word of $w$ in $\uu$.
Derived words to prefixes $w$ of $\uu$ were introduced by Durand in \cite{Durand98} in order to characterize primitively substitutive infinite words.

Among other, Durand showed that if $w$ is a prefix of an infinite word $\uu$ fixed by a primitive substitution, then $\dd_{\uu}(w)$ is fixed by a primitive substitution as well.
Taking all such prefixes $w$, the set of derived words $\dd_{\uu}(w)$ is finite, and thus the set of primitive substitutions fixing these derived  words to prefixes  is finite.  Moreover, if we consider  derived words  to a prefix of a derived  word $\dd_{\uu}(w)$, we obtain again a derived word to some prefix $w'$ of the original word $\uu$.
Thus, we observe that the finite set of primitive substitutions fixing the derived words to prefixes  is invariant under taking derived  word to a prefix and considering its fixing primitive substitution.

If $w$ is not a prefix, then $\dd_\uu(w)$ need not be fixed by a substitution at all.
To study this phenomenon, we introduce the following definition.

\begin{defi} Let $M$ be a finite set of primitive substitutions.
The set  $M$ is said to be  \emph{closed under  derivation}  if the derived word ${\bf d_{\bf u}}(w)$ with respect to any factor $w$ of any fixed point ${\bf u}$ of  $\varphi \in M$  is fixed by a primitive substitution  $\psi \in M$.

 A primitive substitution $\xi$ is called \emph{closeable under derivation} if $\xi$ belongs to a set $M$ closed under derivation.
\end{defi}

Substitutions having Sturmian words as fixed points are very well described (see \cite{Lo2}) and there exists a handy tool to characterize which Sturmian words are fixed by a substitution (see \cite{Ya99}).
It is therefore convenient to start the study of sets closed under derivation by considering Sturmian words and Sturmian substitutions.
In this article, we fully solve this question: in \Cref{neniReflex,jeReflex} we characterize Sturmian substitutions that are closeable under derivation; the characterization is presented in the terms of the representation in the special Sturmian monoid (generated by the morphisms given in \eqref{eq:elem} below) and also alternatively in terms of the slope and the intercept of its fixed point (\Cref{thm:last}).

The article is organized as follows.
\Cref{sec:prelim} contains necessary definitions and notions.
In \Cref{sec:sturm,sec:sturmmor} we introduce needed results on Sturmian words and morphisms.
\Cref{sec:lemmas} contains auxiliary lemmas required in the last \Cref{sec:nonderiv,sec:deriv}, where we deal with the case of Sturmian morphism that are not closed under derivation and Sturmian morphisms that are closed under derivation, respectively.





\section{Preliminaries} \label{sec:prelim}

Let $\A$ denote an \emph{alphabet}, a finite set of symbols called \emph{letters}.
A \emph{finite word} of length $n$ over $\A$ is a concatenation of $n$ letters, i.e., $u=u_0u_1\cdots u_{n-1}$ with $u_i \in \mathcal{A}$.
The \emph{length} of $u$ equals $n$ and is denoted by $|u|$.
The set of all finite words over the alphabet $\A$ and the operation of word concatenation form a monoid $\mathcal{A}^*$.
The unique word of length $0$, the \emph{empty word} $\varepsilon$, is its neutral element.
The \emph{cyclic shift}  of the word $u$ is the word
\begin{equation}\label{eq:def_of_cyc}
{\rm cyc}(u) = u_{1}u_{2}\cdots u_{n-1}u_0.
\end{equation}

An \emph{infinite word} over $\mathcal{A}$ is a sequence $\uu = u_0u_1u_2\cdots  = \left(u_i\right)_{i\in \mathbb{N}} \in \mathcal{A}^{\mathbb{N}}$ with $u_i \in \mathcal{A}$ for all $i \in \N = \left\{ 0,1,2, \ldots \right \}$.
A finite word $w$ is a \emph{factor} of $\uu$ if there exists an integer $i$ such that $w = u_iu_{i+1}u_{i+2}\cdots u_{i+|w|-1}$.
The index $i$ is an \emph{occurrence} of $w$ in $\uu$.
The \emph{language $\mathcal{L}(\uu)$ of $\uu$} is the set of all its factors.
A factor $w$ is a \emph{right special} factor if there exist at least two distinct letters $a,b \in \mathcal{A}$ such that $wa, wb \in \mathcal{L}(\uu)$.
A \emph{left special} factor is defined analogously.
A factor is \emph{bispecial} if it is left and right special.

Given a word $u$, finite or infinite, and finite words $p,v$ and a word $s$ such that $u = pvs$, then we say that $p$ is a \emph{prefix} of $u$ and $s$ is its \emph{suffix}.
The prefix $p$ is \emph{proper} if $p \neq \varepsilon$ and $p \neq u$.

If each factor of $\uu \in \A^\N$ has infinitely many occurrences in $\uu$, the word $\uu$ is \emph{recurrent}.
Given a recurrent infinite word $\uu$ and its factor $w$, a \emph{return word of $w$ in $\uu$} is a factor $v \in \L(\uu)$ such that $vw \in \L(\uu)$ and the factor $w$ occurs in $vw$ exactly twice --- once as a prefix and once as a suffix.
Assume there is an integer $k$ such that $r_0, r_1, \ldots r_k$ are all return words of $w$ in $\uu$.
We can write $\uu = pr_{s_0}r_{s_1}r_{s_2} \ldots$ with $|p|$ equal the least occurrence of $w$ in $\uu$ and $s_i \in \{0,1, \ldots, k\}$.
We say that the word $\left( s_i \right)_{i=0}^{+ \infty}$ is the \emph{derived word of $\uu$ with respect to $w$}, denoted $\dd_\uu(w)$.
For $w$ being a prefix, these words were introduced in \cite{Durand98}.
For a general factor $w$, they are investigated in~\cite{HoZa99}.

In this article, we consider derived words up to a permutation of letters, i.e., we do take into account the indexing of the return words when comparing derived words.




A mapping $\psi: \mathcal{A}^*\to  \mathcal{A}^* $ is a \emph{morphism} over $\mathcal{A}^*$ if $\psi(uv) = \psi(u)\psi(v)$ for each $u,v \in  \mathcal{A}^*$.
The domain of  $\psi$ is extended to $\mathcal{A}^{\mathbb{N}}$ naturally by $\psi(\uu) = \psi(u_0u_1u_2\dots) = \psi(u_0)\psi(u_1)\psi(u_2)\dots$ for $\uu \in \A^\N$.
If  $\psi(\uu) = \uu$, we say that $\uu$ is a \emph{fixed point} of $\psi$.
A morphism $\psi$ is \emph{primitive} if  there exists an integer $k$ such that for each pair of letters $a,b \in \mathcal{A}$ the word  $\psi^k(a)$ contains the letter $b$.

In \cite{Durand98}, a morphism $\psi$ over $\mathcal{A}$ is called \emph{substitution} if there exists a letter $a \in \mathcal{A}$ such that  $\psi(a) =aw$ for some  non-empty word $w$ and the length of the $n^{th}$ iteration of $\psi$ applied to  $a$ tends to infinity, i.e.,  $|\psi^n(a)| \to +\infty$.
Clearly,  any substitution has at least one  fixed point, namely $\uu = aw\psi(w)\psi^2(w)\psi^3(w)\cdots$.
This fixed point is usually denoted as  $\lim_{n\to \infty}\psi^{n}(a)$.
A primitive morphism  $\psi $  has some power $\psi^k$ which is a substitution.
For example,  a morphism given by  $\varphi(0) = 100$ and $\varphi(1)=0$ is not a substitution, but it is primitive as  $ \varphi^2(0) = 0100100$ and $ \varphi^2(1) = 100$.
The morphism $\varphi^2$ is a substitution and has two fixed points, namely $\lim_{n\to\infty}\varphi^{2n}(0)$  and  $\lim_{n\to\infty}\varphi^{2n}(1)$.



If $\uu$ is a fixed point by $\varphi$, then it is also fixed by $\varphi$ for all $k \in \N$.
An infinite word $\uu$ is \emph{rigid} if the set of all morphisms which fix $\uu$ is of the form $\left\{ \varphi^k
\colon k \in \N \right\}$ for some morphism $\varphi$.

\section{Sturmian words} \label{sec:sturm}

Sturmian words are infinite words over a two letter alphabet having the least unbounded factor complexity possible.
In other words, an infinite word is \emph{Sturmian} if for each $n \in \N$ the number of its factors of length $n$ equals $n+1$.
There are many other characterizations of Sturmian words.
For the phenomenon that we investigate, the characterization based on the notion of interval exchange transformation is the most suitable.

For a given parameters $\ell_0, \ell_1 >0$, we  consider the partition of the interval $I=[0,\ell_0+\ell_1)$ into $I_0=[0,\ell_0)$ and $I_1=[\ell_0, \ell_0+\ell_1)$ or the partition of $I=(0,\ell_0+\ell_1]$   into  $I_0=(0,\ell_0]$ and $I_1=(\ell_0, \ell_0+\ell_1]$.
The transformation $T: I\to I$ defined by
$$
T(x) =
\left\{
\begin{array}{ll}
x +  \ell_1 & \text{ if } x \in I_0,\\
x -  \ell_0 & \text{ if }  x \in I_1
\end{array}
\right.
$$
is a \textit{two interval exchange transformation}, or shortly \emph{2iet}.
If we take an initial point $ \rho \in I$, the sequence $\uu = u_0u_1u_2 \dots \in \{0,1\}^\N$ defined by
$$
u_n =
\left\{
\begin{array}{ll}
0 & \text{ if} \ T^n(\rho) \in  I_0,\\
1 & \text{ if} \ T^n(\rho)  \in I_1
\end{array}
\right.
$$
is a \textit{2iet sequence} with the \textit{parametrs}  $\ell_0,\ell_1,\rho$.
In other words, a 2iet sequence is a coding of itineraries $\left( T^n(\rho) \right)_{n=0}^{+\infty} $ with respect to the partition $I_0 \cup I_1$.
The value  $ \gamma = \frac{\ell_1}{\ell_0+\ell_1}$  is called the \emph{slope} of  $\uu$.
It is well known that the set of all 2iet sequences having  an irrational slope  coincides with the set of all Sturmian words (see for instance \cite{Lo2}).
If  we need  to distinguish   whether a Sturmian word comes from a transformation with the domain  $I=[0,\ell_0+\ell_1)$ or with the domain $I=(0,\ell_0+\ell_1]$ we use the names \emph{lower and upper Sturmian word}, respectively.
For most of the parameters $ \rho \in (0, \ell_0+\ell_1)$ the lower Sturmian word with the parameters $\ell_0,\ell_1,\rho$ equals to  the upper Sturmian word with the same parameters.

 Clearly, lower (upper)  Sturmian words  corresponding to the  triplets $(\ell_0,\ell_1,\rho)$ and $(c\ell_0,c\ell_1,c\rho)$ coincide for any positive constant $c$.
It is the reason for the triplet of parameters  $\ell_0,\ell_1,\rho$ to be often normalized into the form $$\Bigl(\frac{\ell_0}{\ell_0+\ell_1}, \frac{\ell_1}{\ell_0+\ell_1},\frac{\rho}{\ell_0+\ell_1}\Bigr)=(1-\gamma,\gamma,\delta), $$ where $\gamma$ is the slope.
 The lower Sturmian  word with parameters $(1-\gamma,\gamma,\delta) $ where   $\delta \in [0,1)$ is in \cite{Lo2} denoted by $\ss_{\gamma, \delta }$ and
  the upper  Sturmian  word with parameters $(1-\gamma,\gamma,\delta) $ where  $\delta  \in (0,1]$ is  denoted by
 ${{\ss^\prime}}_{\gamma, \delta}$.

%
%

The language $\L(\uu)$ of a Sturmian word $\uu$ is independent of the parameter   $\rho$,  it depends only on the slope $\gamma=\frac{\ell_1}{\ell_0+\ell_1}$.
Any Sturmian word is uniformly recurrent.  Frequencies  of the letters $0$ and $1$ are $1-\gamma$ and $\gamma$,  respectively.
Among all Sturmian words with a  fixed irrational slope $\gamma=\frac{\ell_1}{\ell_0+\ell_1}$,   the sequence with the triplet of parameters
$(\ell_0, \ell_1, \ell_1)$   plays a special role.
Such a sequence is called a \textit{standard Sturmian word} and it is usually denoted by  ${\bf c}_\gamma$.
Any prefix of ${\bf c}_\gamma$ is a left special factor.
In other words, a Sturmian word $\uu \in \{0,1\}^\N$ is standard if both sequences $0\uu$, $1\uu$ are Sturmian.

The \emph{shift operator} $\sigma$ maps an infinite word $\uu = u_0u_1u_2\dots$ to the word $\sigma(\uu) = u_1u_2u_3\dots$, i.e., $\sigma$ erases the starting letter of the word $\uu$. If $\uu$  is a Sturmian word coding an initial point $\rho$  under a two interval exchange transformation $T$, then   $\sigma(\uu)$  is coding of the initial point  $T(\rho)$.

\begin{obs}\label{shift}
If $\uu$ is  a lower Sturmian word with parameters $\ell_0,\ell_1$, and $\rho$, then  $\sigma(\uu)$  is a lower Sturmian word  with parameters $\ell_0,\ell_1$, and $\rho'$, where
\[
\rho'= \begin{cases}  \rho +\ell_1 & \text{ if }  \rho \in [0,\ell_0), \\
\rho -\ell_0 & \text{ if } \rho \in [\ell_0, \ell_0+\ell_1 ).
\end{cases}
\]

If $\uu$  is  an upper  Sturmian word with parameters $\ell_0,\ell_1$, and $\rho$, then  $\sigma(\uu)$  is an upper  Sturmian word  with parameters $\ell_0,\ell_1$, and $\rho'$, where
\[
\rho'= \begin{cases}  \rho +\ell_1 & \text{ if }  \rho \in (0,\ell_0], \\
\rho -\ell_0 & \text{ if } \rho \in (\ell_0, \ell_0+\ell_1].
\end{cases}
\]
\end{obs}

\section{Sturmian morphisms} \label{sec:sturmmor}

In this article, we work with these four elementary morphisms:
\begin{equation} \label{eq:elem}
\varphi_a: \begin{cases} 0 \to 0 \\ 1 \to 10 \end{cases}, \quad \quad
	\varphi_b: \begin{cases} 0 \to 0 \\ 1 \to 01 \end{cases}, \quad \quad
	\varphi_\alpha: \begin{cases} 0 \to 01 \\ 1 \to 1 \end{cases}, \quad \quad
	\varphi_\beta: \begin{cases} 0 \to 10 \\ 1 \to 1 \end{cases}.
\end{equation}
Each of these 4 morphisms is a so-called Sturmian morphism, that is a morphism such that any its image of a Sturmian word is a again a Sturmian word.
Moreover, the morphisms $\varphi_b$ and $\varphi_\beta$ map a standard Sturmian word to a standard Sturmian word and are thus called \emph{standard Sturmian morphisms}.
Let $\mathcal{M}$ be the monoid generated by the four morphisms, i.e. $\mathcal{M} = \langle \varphi_a, \varphi_b, \varphi_\alpha, \varphi_\beta \rangle $.  The monoid $\mathcal{M}$ is usually called the \emph{special Sturmian monoid}.
For a non-empty word $w = w_0\cdots w_{n-1}$ over the alphabet $\{a,b,\alpha,\beta\}$ we set
\[
 \varphi_w = \varphi_{w_0} \varphi_{w_1} \cdots \varphi_{w_{n-1}}.
\]
Each morphism $\varphi_w$ maps a lower (upper) Sturmian word to a lower (upper) Sturmian word.  A morphism $\varphi_w$  is  primitive if and only if $w$ contains at least one Latin letter and at least one Greek letter.
If $\varphi_w$ is primitive, then $\varphi_w$ is a substitution.
To obtain the monoid of all Sturmian morphisms we have to extend the set of generators by the morphism $E: 0 \mapsto 1, 1 \mapsto 0 $.  This morphism maps a lower (upper) Sturmian word to an upper (lower) Sturmian word. If $\psi$ is a Sturmian morphism from the monoid $\langle E, \varphi_a, \varphi_b, \varphi_\alpha, \varphi_\beta \rangle $, then $\psi^2\in \mathcal{M}$.

The four elementary morphisms  $ \varphi_a, \varphi_b, \varphi_\alpha, \varphi_\beta$ serve as a basis for a representation of any Sturmian word.

\begin{thm}[\cite{JuPi}] \label{Sadic}
An infinite binary word $\uu$ is Sturmian if and only if there exists an infinite word ${\bf w}=w_0w_1w_2\cdots$ over the alphabet $\{a,b,\alpha, \beta\}$ and an infinite sequence $(\uu_i)_{i\geq 0}$ of Sturmian words such that  $\uu = \uu_0$ and
$\uu_{i} = \varphi_{w_i}(\uu_{i + 1})$ for all $i \in \mathbb{N}$.
\end{thm}

Usually, the sequence $(\varphi_{w_i})$ is called an \emph{ S-adic representation} of $\uu$.
In our context, the sequence $(w_i)$ shall be simply called an S-adic representation of $\uu$.

The monoid $\mathcal{M}$ is a proper submonoid of the monoid of all Sturmian morphisms and it is not free.
It is easy to show that for any $k \in \N$ we have
\[
\varphi_{\alpha a^k\beta } = \varphi_{\beta b^k\alpha}\quad \text{ and } \quad \varphi_{a\alpha^kb} = \varphi_{b\beta^ka}.
\]
In fact, these rules give the presentation of the monoid:
\begin{thm}[\cite{See91,ReKa07}] \label{thm:relations}
Let $w,v\in \{a,b,\alpha,\beta\}^*$.
The morphism $\varphi_w$ equals $\varphi_{v}$ if and only if the word $v$ can be obtained from $w$ by possibly repeated application of the rewriting rules
\begin{equation}\label{eq:relations}
  \alpha a^k\beta = \beta b^k\alpha \quad \text{ and } \quad  a\alpha^kb = b\beta^ka\qquad \text{ for any $k \in \N$ }.
\end{equation}
\end{thm}

Note that the rules~\eqref{eq:relations} preserve  positions in $w \in  \{a,b,\alpha,\beta\}^*$ of Latin and Greek letters.
Thus, by setting $a<b$ and $\alpha < \beta$ we may define a lexicographic order on all equivalent words in $\{a,b,\alpha,\beta\}^*$.
 In~\cite{KlMePeSt18}, this lexicographic order allowed to describe the derived words with respect to the prefixes of Sturmian words.
 It is the reason to use a non-traditional notation of the elementary morphisms, which are in~\cite{Lo2} denoted as follows:
\begin{equation}\label{LothairJinak}
\varphi_b = G, \quad  \varphi_a = \widetilde{G}, \quad \varphi_\beta = D, \quad  \varphi_\alpha = \widetilde{D}.
\end{equation}
\begin{defi}
Let $w \in  \{a,b,\alpha,\beta\}^*$.
The lexicographically largest word in $  \{a,b,\alpha,\beta\}^*$ which  can be obtained from $w$ by application of rewriting rules \eqref{eq:relations} is denoted $N(w)$.
If  $\psi = \varphi_w$,  then the word $N(w)$ is the \emph{normalized name} of the morphism $\psi$  and it is also denoted by $N(\psi) = N(w)$.
\end{defi}

The next lemma is a direct consequence of Theorem \ref{thm:relations}.
\begin{lem}\label{lem:normalized_words}
Let $w  \in \{a, b, \alpha, \beta\}^*$.
We have $w = N(w)$ if and only if $w$ does not contain $\alpha a^k\beta$ or $a\alpha^kb$ as a factor for any $k \in \mathbb{N}$.
In particular,  if $w \in \{ a, b, \alpha, \beta\}^* \setminus \{a, \alpha\}^*$, the  normalized name $N(w)$ has prefix either $a^i\beta$ or $\alpha^ib$ for some $i \in \mathbb{N}$.
\end{lem}

The following definition is helpful in the description of derived words with respect to prefixes.

\begin{defi} \label{def:delta}
	Let $w\in  \{a,b,\alpha,\beta\}^*\setminus \{a,\alpha\}^*$ be the normalized name of a morphism $\psi$. Put
	\[
		\Delta(w) =
			\begin{cases}
				N( w'a^k\beta) & \text{ if \ } w = a^k\beta w', \\
				N(w'\alpha^kb) & \text{ if \ } w = \alpha^kb w'
			\end{cases}
	\]
	with $k \in \N$.
	For $\psi = \varphi_w$ we set
	\[
		\Delta (\psi) = \varphi_{\Delta(w)}.
	\]
\end{defi}

\begin{example} \label{ex:N}
Let $v = aa\beta \alpha \beta \beta a$.
We have $N(v) = aa\beta\beta\beta\alpha a$ and $\Delta(N(v)) = N(\beta\beta\alpha aa\beta\beta) = \beta\beta \beta bb \beta \alpha $.
\end{example}

\begin{remark}\label{vsechny4}
Let us point out several properties of the operation $\Delta$.
Assume $w \in \{a,b,\alpha,\beta\}^*\setminus \{a,\alpha\}^*$ such that $\varphi_w$ is primitive.

\begin{enumerate}[(I)]
\item \label{it:vsechny4:1} If only two letters from $\{a,b,\alpha,\beta\}$  occur in  $w$,  i.e., $w \in \{b,\beta\}^* \cup  \{b,\alpha\}^*\cup
 \{a,\beta\}^*$,  then $\Delta(w)$ is  an iteration of the cyclic shift.

\item \label{it:vsechny4:2}
  The function $\Delta$ preserves the length of a word.
  The number of the letters $b$ and $\beta$ in the  normalized name   $N(w)$  of a word $w$ is  never smaller than the number of these letters in $w$.
  Thus the sequence  $\Delta^i(w)$  of iterations of $\Delta$  is eventually periodic and for each sufficiently large index $i$, $\Delta^{i+1}(w)$ of can be computed from $\Delta^i(w)$  by cyclic shift without using normalization.

 \item \label{it:vsechny4:3} If $w$ contains at least one letter from $\{a, \alpha\}$,  then  the form of rewriting rules  \eqref{eq:relations} implies that  $N(w)$ and $\Delta(w)$  contains at least one letter from  $\{a, \alpha\}$.

\item \label{it:vsechny4:4}
 If $w$  contains at least 3 letters from   $\{a,b,\alpha,\beta\}$, then  by Example 31 of \cite{KlMePeSt18}, $\Delta^i(w)$ contains both letters $\beta$  and $b$ for each sufficiently large  $i$.
\end{enumerate}
\end{remark}
\begin{thm}[\cite{KlMePeSt18}]\label{thm:main_result_vetsina}
	Let $\psi \in \langle\varphi_a, \varphi_b, \varphi_\alpha, \varphi_\beta\rangle$ be a primitive morphism and $ N(\psi) = w \in \{a,b,\alpha,\beta\}^* \setminus \{a, \alpha\}^*$ be its normalized name.
	If $\uu$ is the fixed point of $\psi$,
	then $\mathbf{x}$ is (up to a permutation of letters) a derived word of $\uu$ with respect to one of its prefixes if and only if $\mathbf{x}$ is the fixed point of the morphism $\Delta^j(\psi)$ for some $j\in \mathbb{N}$.
\end{thm}

\begin{example}[\Cref{ex:N} continued]
Taking $w = aa\beta\beta\beta\alpha a$, we have
\[
\begin{aligned}
\Delta(w) & = \beta\beta \beta bb \beta \alpha, & \Delta^7(w) & =   \beta \beta \beta \alpha b b \beta, \\
\Delta^2(w) & =  \beta\beta  bb \beta \beta \alpha, & \Delta^8(w) & =   \beta \beta \alpha b b \beta \beta, \\
\Delta^3(w) & =  \beta  bb \beta \beta \beta \alpha, & \Delta^{9}(w) & =   \beta \alpha b b \beta \beta \beta, \\
\Delta^4(w) & =  bb \beta \beta \beta \beta \alpha, & \Delta^{10}(w) & =   \alpha b b \beta \beta \beta \beta, \\
\Delta^5(w) & =  b \beta \beta \beta \beta \alpha b, & \Delta^{11}(w) & =   b \beta \beta \beta \beta \alpha b = \Delta^5(w). \\
\Delta^6(w) & =  \beta \beta \beta \beta \alpha b b,
\end{aligned}
\]
Note that for $i > 4$, the element $\Delta^i(w)$ is always a cyclic shift of the previous element $\Delta^{i-1}(w)$, illustrating \Cref{it:vsechny4:2,it:vsechny4:4} of \Cref{vsechny4}.
\end{example}

In \cite{KlMePeSt18}, we considered only derived words to non-empty prefixes.
If we include in our considerations also the empty prefix $\varepsilon$, then the derived word to $\varepsilon$ in $\uu$ is $\uu$ itself and it is fixed by $\psi = \Delta^{0}(\psi)$.

The objective of this article is to detect the sets of primitive Sturmian morphisms that are closed under derivation.
In order to do that, a tool deciding whether a Sturmian word is fixed by such a morphism is needed.
It is easy to see that a Sturmian word $\uu$ is  fixed by  a primitive morphism if and only if $\uu$ has a purely periodic $S$-adic representation.
Yasutomi \cite{Ya99}  found a characterization of such Sturmian words using algebraic properties of their parameters.
To quote his result we recall that  a number $\lambda$  is quadratic if it is an irrational root of a quadratic equation $Ax^2+Bx+C=0$ with rational coefficients $A\neq 0, B,C$.

Let $\mathbb{Q}(\lambda)$ denote the minimal number field containing $\mathbb{Q}$ and $\lambda$.
If  $\lambda$  is quadratic,  then $\mathbb{Q}(\lambda) = \{c+d\lambda \colon c,d \in \mathbb{Q} \}$.
Let $\overline{\lambda}$ be the other root of $Ax^2+Bx+C=0$, i.e., the algebraic conjugate of $\lambda$.
Since the mapping $z  = c+d\lambda\mapsto \overline{z} = c+d\overline{\lambda }$  is an automorphism of the field $\mathbb{Q}(\lambda)$, we have $ \overline{z+y}= \overline{z}+\overline{y}$ and $ \overline{z\cdot y}= \overline{z}\cdot \overline{y}$ for each  $z,y \in \mathbb{Q}(\lambda)$.

\begin{thm}[\cite{Ya99}]\label{Yasutomi0}
Let $ \gamma, \delta \in [0,1]$ and $\gamma$ be irrational.
A Sturmian word coding the two interval exchange transformation with parameters  $ \ell_0=1-\gamma,\ell_1=  \gamma, \rho = \delta $
 is fixed by a primitive morphism if and only if
\begin{enumerate}
\item $\gamma $  and $\delta$ belong to the same quadratic field $\Q(\gamma)$; and
\item $\overline{\gamma }\notin (0,1)$; and
\item  If $\overline{\gamma }>1$,  then  $\overline{\delta} \in [1-\overline{\gamma },\overline{\gamma}]$; \ \ if  $\overline{\gamma }<0$, then $\overline{\delta} \in [ \overline{\gamma}, 1-\overline{\gamma }]$.
\end{enumerate}
\end{thm}

A quadratic number  $\gamma\in (0,1)$  with  conjugate $ \overline{\gamma} \notin (0,1)$ is called a \emph{Sturm number}.
 The notion  Sturm number  was originally defined via properties of coefficients in the continued fraction expansion of $\gamma$, later Allauzen \cite{Allauzen} found an algebraic characterization of Sturm numbers.

The parameters  $(\ell_0,\ell_1,\rho)$  of a Sturmian word  in the previous theorem satisfy   $\ell_0+\ell_1 = 1$, i.e.,  the parameter
$\ell_1$ equals the slope.
 We rewrite  this theorem to a  form  which is more convenient for our considerations.
 We normalize the parameters  $\ell_0,\ell_1$  of a two interval exchange $T$  to satisfy the condition that the longer interval is of length 1 and the shorter one is of length $\theta$. Clearly,  $\theta \in (0,1)$ and the slope $\gamma$ equals $\tfrac{\theta}{1+\theta}$ or $\tfrac{1}{1+\theta}$.
 This kind of normalization is also used in \cite[Chapter 6]{Fogg} in order to reveal the relation of Sturmian words to the Ostrowski numeration system.

\begin{thm}\label{Yasutomi}
Let $\theta \in (0,1)$ be irrational and $\rho \in [0,1+ \theta]$.    A Sturmian word with parameters   $ 1, \theta,\rho$ or
$\theta,1, \rho$ is fixed by a primitive morphism if and only if
\begin{enumerate}[(1)]
\item $\theta$  and $\rho$ belong to the same quadratic field; and
\item $\overline{\theta} <0$; and
\item \label{it:Yasutomi:3} $\overline {\theta} \leq \overline{\rho}\leq 1$.
\end{enumerate}
\end{thm}
\begin{proof}
The relation between the parameters $\gamma$ and $\delta$ in \Cref{Yasutomi0} and the parameters in our modification is $\gamma= \tfrac{\theta}{1+\theta}$ or $\gamma= \tfrac{1}{1+\theta}$ and $\delta = \tfrac{\rho}{1+\theta}$.
Equivalently, $$\Bigl(\theta= \tfrac{1-\gamma}{\gamma}\ \ \text{and}  \ \ \rho = \tfrac{\delta}{ \gamma}\Bigr)   \quad  \text{or }  \quad
\Bigl(\theta= \tfrac{\gamma}{1-\gamma}\ \ \text{and}  \ \  \rho = \tfrac{\delta}{1-\gamma}\Bigr)\,.$$
Clearly, $\gamma$ and $\delta$ belong to the same quadratic field if and only if $\theta$ and $\rho$ belong to the same quadratic field.
The fact that the quadratic slope $\gamma$ is a Sturm number implies $$ \overline{\gamma}\notin (0,1) \ \ \Longleftrightarrow \ \  \overline{\gamma}(1-\overline{\gamma})<0 \ \ \Longleftrightarrow \ \ \tfrac{\overline{\gamma}}{1-\overline{\gamma}}<0 \quad \text{ and } \quad \tfrac{1-\overline{\gamma}}{\overline{\gamma}} <0
\ \  \Longleftrightarrow \ \  \overline{\theta} <0 \,.
$$
If $\overline{\gamma} >1$, \Cref{it:Yasutomi:3} of \Cref{Yasutomi} can be equivalently rewritten
$$1-\overline{\gamma }\leq \overline{\delta}\leq  \overline{\gamma}\ \ \Longleftrightarrow \ \
1\geq  \tfrac{\overline{\delta}}{1-\overline{\gamma}}\geq \tfrac{ \overline{\gamma}}{1-\overline{\gamma}}  \quad \text{ and } \quad
\tfrac{ 1-\overline{\gamma}}{\overline{\gamma}}\leq  \tfrac{\overline{\delta}}{\overline{\gamma}}\leq  1
\ \ \Longleftrightarrow \ \   1\geq \overline{\rho} \geq \overline{\theta}\,.$$
If $\overline{\gamma} < 0$, then
\[\overline{\gamma }\leq \overline{\delta}\leq  1-\overline{\gamma}\ \ \Longleftrightarrow \ \  \tfrac{ \overline{\gamma}}{1-\overline{\gamma}} \leq
 \tfrac{\overline{\delta}}{1-\overline{\gamma}}\leq \ 1   \quad \text{ and } \quad 1  \geq
 \tfrac{\overline{\delta}}{\overline{\gamma}}\geq \ \tfrac{1- \overline{\gamma}}{\overline{\gamma}}  \ \ \Longleftrightarrow \ \  \overline{\theta}\leq \overline{\rho}  \leq 1. \qedhere
\]
\end{proof}

\section{Auxiliary lemmas} \label{sec:lemmas}

\begin{lem}\label{prolong}
Let $w$ be a factor of an aperiodic word $\uu$.
\begin{enumerate}[(1)]
\item \label{it:prolong1} There exists $s$ such that $ws$ is  right special  in  $\uu$, and $ws'$ is not right special for any proper  prefix  $s'$ of $s$.
Moreover,  $\dd_{\uu}(w) =\dd_{\uu}(ws)$.
\item \label{it:prolong2} There exists $p$ such that $pw$ is  left special  in  $\uu$, and $p'w$ is not left special for any proper suffix   $p'$ of $u$.
 Moreover,  if $p'w$ is a prefix of $\uu$ for some proper  suffix $p'$ of $p$, then  $\dd_{\uu}(w) =\dd_{\uu}(p'w)$.
 Otherwise, $\dd_{\uu}(w) =\dd_{\uu}(pw)$.
\end{enumerate}
\end{lem}

\begin{proof}
\Cref{it:prolong1}:
As $\uu$ is aperiodic, the factor $w$ is a factor of some right special factor.
Let $ws$ be the shortest such right special factor.
Thus, the word $s$ is unique.
Let $r_1,r_2,\ldots,r_k$ be all the return words of $w$.
As the factor $s$ always occurs after $w$, the word $r_iws \in \L(\uu)$.
Since $w$ is a prefix of $r_iw$, the word $ws$ is a prefix of $r_iws$, and so $r_i$ is a return word of $ws$.
To conclude, the return words of $w$ and $ws$ are identical, and thus so are the derived words with respect to $w$ and $ws$.

\Cref{it:prolong2}:
The factor $w$ is a factor of some left special factor.
Let $pw$ be the shortest such left special factor.
Let $r_1,r_2,\ldots,r_k$ be all the return words of $w$.
Similarly to the previous case, the word $p'r_iw$ contains exactly two occurrences of $p'w$ for each prefix $p'$ of $p$.
Thus, $p'r_i$ is a return word of $p'w$.
The occurrences of $r_i$ are occurrences of $p'r_i$ shifted by $p'$ for all $i$ maybe except for the first occurrence.
Thus, if $p'w$ is a prefix of $\uu$ for some suffix of $p'$ of $p$, then $\dd_{\uu}(w) =\dd_{\uu}(p'w)$.
In the other case, $\dd_{\uu}(w) =\dd_{\uu}(pw)$.
\end{proof}

As a consequence of the last lemma, to describe the derived words to all factors of $\uu$  we can restrict our study to factors $w$ such that $w$ is either a right special prefix of $\uu$  or $w$ is  a bispecial factor of $\uu$ but  not a prefix of $\uu$.

The next proposition is  borrowed from our previous article, where  we investigated  derived words to prefixes of Sturmian words.
\begin{prop}[\cite{KlMePeSt18}]\label{prop:der_of_preimage_fi_b}
	Let $\uu$  and $\uu'$ be Sturmian words such that $\uu = \varphi_{b}(\uu')$ and let $w'$ be a non-empty right special prefix of $\uu'$.
	We have $d_{\uu'}(w') = d_{\uu}(w)$ with $w =  \varphi_{b}(w')0$.
\end{prop}

\begin{coro}\label{jinak}  Let $\ww$ be a Sturmian word and $k\in \mathbb{N}$.
\begin{enumerate}[(1)]
	\item \label{Claim1}
Denote    $\uu= \varphi^k_b(\ww)$.
The word $0^k$ is a bispecial prefix of $\uu$ and  $\dd_\uu(0^k) = \ww$.
\item \label{Claim2}
If $1\ww$ is  a Sturmian word, then  $\varphi^k_a(1\ww) =1\varphi^k_b(\ww)$.
\end{enumerate}
\end{coro}

\begin{proof}
\Cref{Claim1}: The proof is done by induction on $k$.
If $k=0$, then the empty word is right special factor of $\uu=\ww$ and the derived word with respect to the empty word is the word $\ww$ itself.

Now   $k>1$.  First, write the explicit form of the morphism  $\varphi_b^k: 0\mapsto 0\text{ and } 1\mapsto 0^k1$. Therefore, in the word $\uu$ two neighbouring occurrences of the letter 1 are  separated either by the  block $0^k$ or  $0^{k+1}$. It means that $0^k$ is a bispecial factor of $\uu$.  Let us assume that  $ \uu' = \varphi^{k-1}_b(\ww)$ has a right special prefix $0^{k-1}$  and   $\dd_{\uu'}(0^{k-1}) = \ww$.  Clearly, $0^k = \varphi_b(0^{k-1})0$  is a right special prefix of  $\uu=\varphi_b(\uu') =  \varphi^k_b(\ww)$.
By Proposition  \ref{prop:der_of_preimage_fi_b},
  $\dd_{\uu}(0^{k})=  \dd_{\uu'}(0^{k-1}) =  \ww$.

\Cref{Claim2}: We proceed by induction on $k$. The case  $k=0$ is trivial.  If $k>0$,
we use the fact that the morphisms $\varphi_a$ and $\varphi_b$ are conjugate, in particular
$ 0\varphi_a(x) = \varphi_b(x)0 \  \text{ for every }\ x \in \{0,1\}^*$.
It implies $ 0\varphi_a({\bf x}) =  \varphi_b({\bf x})$ for every ${\bf x} \in \{0,1\}^\mathbb{N}$.

Applying this property to the word $1\ww$  and using the induction hypothesis, we obtain
$\varphi^k_a(1\ww) = \varphi_a\bigl(\varphi^{k-1}_a(1\ww)\bigr) = \varphi_a\bigl(1 \varphi_b^{k-1}(\ww) \bigr) = \varphi_a(1)\bigl(\varphi_b^{k-1}(\ww) \bigr) = 10 \varphi_a\bigl(\varphi_b^{k-1}(\ww) \bigr)  = 1 \varphi_b\bigl(\varphi_b^{k-1}(\ww) \bigr) =1\varphi_b^{k}(\ww)$.
\end{proof}

\begin{lem}\label{image}
Let $\uu$ be a Sturmian word with parameters $\ell_0$, $\ell_1$  and $\rho$.
The Sturmian word
\begin{itemize}
\item $\varphi_b(\uu)$ has parameters $\ell_0+\ell_1$, $\ell_1$ and $\rho $;
\item  $\varphi_a(\uu)$ has parameters  $\ell_0+\ell_1$, $\ell_1$ and  $\rho +\ell_1$;
\item $\varphi_\beta(\uu)$ has parameters $\ell_0$, $\ell_0+\ell_1$ and  $\rho +\ell_0$;
\item  $\varphi_\alpha(\uu)$ has parameters $\ell_0$, $\ell_0+\ell_1$ and  $\rho $.
\end{itemize}
\end{lem}
\begin{proof}
\cite[Lemma 2.2.18]{BeSe_Lothaire} claims that $G= \varphi_b$  maps the lower Sturmian word $\ss_{\gamma,\delta}$ to the word
$\ss_{\frac{\gamma}{1+\gamma},\frac{\delta}{1+\gamma}}$. And analogously, the upper Sturmian word
 $\ss'_{\gamma,\delta}$ is mapped by $G$ to to the upper Sturmian word
$\ss'_{\frac{\gamma}{1+\gamma},\frac{\delta}{1+\gamma}}$.
Thus, using our notation,  a  Sturmian word with the triplet of parameters
$\frac{1}{\ell_0 + \ell_1} \left (\ell_0, \ell_1,\rho \right ) = (1-\gamma,
\gamma, \delta)$ is mapped by $\varphi_b$ to a Sturmian word with the triplet of  parameters
$$ (\ell_0^{new}, \ell_1^{new} ,\rho^{new})=c(1- \tfrac{\gamma}{1+\gamma}, \tfrac{\gamma}{1+\gamma}, \tfrac{\delta}{1+\gamma}),$$ where $c$ is an arbitrary positive constant.
If we choose  $c ={1+\gamma}$,  we obtain the triplet
$$ \ell_0^{new}=1 =\ell_0+\ell_1, \quad    \ell_1^{new}=\gamma =\ell_1  \quad \text{and} \quad  \rho^{new}=\delta = \rho.$$
Proof of the remaining part of the lemma is analogous.
\end{proof}
\begin{lem}\label{interceptzustava}
Let the two word $u= u_1u_2\cdots u_n\in \{a,b,\alpha,  \beta\}^*$ and  $v= v_1v_2\cdots v_n\in \{a,b,\alpha,  \beta\}^*$ satisfy for each $k=1,2,\ldots,n$\ :
\begin{equation}\label{podobne}
\text {if } v_k \in \{a,b\} \text{, then } u_k\in \{a,b\}  \quad \quad \text{ and } \quad \quad \text{if } v_k \in \{\alpha,\beta\} \text{, then } u_k\in \{\alpha,\beta\}.
\end{equation}
If $\uu$ and $\vv$ are Sturmian words with the same slope, then the slopes of   $\varphi_u(\uu)$  and $\varphi _v(\vv)$ are equal.
\end{lem}
\begin{proof} By \Cref{image}, both morphisms $\varphi_b$ and $\varphi_a$  change the original slope
 $\tfrac{\ell_1}{\ell_0+\ell_1}$ to the same new slope   $\tfrac{\ell_1}{\ell_0+2\ell_1}$. And analogously, both morphisms $\varphi_\beta$ and $\varphi_\alpha$  change the  original slope
 $\tfrac{\ell_1}{\ell_0+\ell_1}$ to the same new slope  $\tfrac{\ell_1}{2\ell_0+\ell_1}$.
\end{proof}

\section{Sturmian morphisms non-closeable under derivation} \label{sec:nonderiv}

The following example shows that not every Sturmian morphism is closeable under derivation.

\begin{example}
Consider a Sturmian morphism $\Psi = \varphi_{ab\beta}$.
We have
\[
\Psi: \begin{cases}
0 \mapsto 100,\\
1 \mapsto 10010.
\end{cases}
\]
Let $\uu$ be its fixed point:
\[
\uu = 1001010010010010100 \dots .
\]
We claim that $\uu$ is a Sturmian word with parameters $\vec{x} = \left( \sqrt{3}-1, 2-\sqrt{3}, \frac{3-\sqrt{3}}{2} \right) $.

Let $\vv$ be the lower Sturmian word with parameters $\vec{x}$.
By \Cref{image},
\begin{itemize}
\item $\varphi_{\beta}(\vv)$ has parameters $\left( \sqrt{3}-1, 1, \frac{1+\sqrt{3}}{2} \right) $,
\item $\varphi_{b\beta}(\vv)$ has parameters $\left( \sqrt{3}, 1, \frac{1+\sqrt{3}}{2} \right) $,
\item $\varphi_{ab\beta}(\vv)$ has parameters $\left( \sqrt{3}+1, 1, \frac{3+\sqrt{3}}{2} \right) = \frac{1}{2-\sqrt{3}} \vec{x}$.
\end{itemize}
Thus, $\vv$ is fixed by $\Psi$, and so $\uu = \vv$.

Next, we show that the derived word with respect to $0 \in \L(\uu)$ is not fixed by any primitive substitution, which implies that $\Psi$ is not closeable under derivation.

The factor $0$ is not a prefix of $\uu$.
It is a prefix of $\sigma(\uu)$.
By~\Cref{shift}, the word $\sigma(\uu)$ has parameters $\left( \sqrt{3}-1, 2 - \sqrt{3}, \frac{7-3\sqrt{3}}{2} \right) $.
The return words to $0$ in $\uu$ (and $\sigma(\uu)$) are $r_0 = 0$ and $r_1 = 01$.
Thus, we may write
\[
\sigma(\uu) = r_0r_1r_1r_0r_1r_0r_1r_0r_1r_1r_0 \dots .
\]
Since $r_0 = \varphi_b(0)$ and $r_1 = \varphi(1)$, we obtain
\[
\sigma(\uu) = \varphi_b(\dd_\uu(0)).
\]
By \Cref{image}, the derived word $\dd_\uu(0)$ has parameters $\left( 2\sqrt{3}-3, 2-\sqrt{3},\frac{7-3\sqrt{3}}{2} \right) $.
In order to use \Cref{Yasutomi0}, we normalize the parameters to $\frac{1}{\sqrt{3}-1}  \left( 2\sqrt{3}-3, 2-\sqrt{3},\frac{7-3\sqrt{3}}{2} \right) $.
Using the notation of \Cref{Yasutomi0}, we have $\gamma = \frac{2-\sqrt{3}}{\sqrt{3}-1} = \frac{\sqrt{3}-1}{2}$ and $\rho = \frac{7-3\sqrt{3}}{2(\sqrt{3}-1)} = \frac{2\sqrt{3}-1}{2}$.
Considering the algebraic conjugates $\overline{\gamma} = \frac{-\sqrt{3}-1}{2}$ and $\overline{\rho} = \frac{-2\sqrt{3}-1}{2}$, we notice
\[
\overline{\gamma} < 0 \quad \text{ and } \quad \overline{\rho} < \overline{\gamma}.
\]
Therefore, the third condition of \Cref{Yasutomi0} is not satisfied, and thus $\dd_\vv(0) = \dd_\uu(0)$ is not fixed by a primitive substitution.
\end{example}

In the general case, we prove later the following theorem on Sturmian substitutions that are not closeable under derivation.

\begin{thm}\label{neniReflex}
Let $\psi= \varphi_w$ be a Sturmian morphism such that  $w \in \{a,b,\alpha,\beta\}^*$.
If at least three distinct letters from  $ \{a,b,\alpha,\beta\}$ occur in $w$, then $\psi$ is not closeable under derivation.
\end{thm}

First we prepare several auxiliary statements exploited in the proof of the above theorem.

\begin{lem}\label{vRovinePulka1}
Let  a Sturmian word   $\uu$  with parameters  $(\ell_0, \ell_1,\rho)$  be fixed   by a primitive morphism $\varphi_w$.
\begin{enumerate}[(1)]
\item   If  $w \in \{ b,\beta\}^*$, then   $\rho = \ell_1$; \label{it:vRovinePulka1}
\item If  $w \in \{ b,\alpha\}^*$,  then $\rho = 0$; \label{it:vRovinePulka2}
\item  If  $w \in \{ a,\beta\}^*$, then $\rho =\ell_0+\ell_1 $; \label{it:vRovinePulka3}
\item  If  $w \in \{ a,\alpha\}^*$, then  $\rho = \ell_0$. \label{it:vRovinePulka4}
 \end{enumerate}
\end{lem}
\begin{proof}
We define four planes
\[
\begin{aligned}
P_1 &=  \left \{(x,y,z) \in \R^3 \colon  z=y \right\}, & \quad P_2 &= \left \{(x,y,z) \in \R^3 \colon z=0 \right\}, \\
P_3 &= \left \{(x,y,z) \in \R^3 \colon  z=x+y   \right\}, &  \text{ and } \quad  P_4 &= \left \{(x,y,z) \in \R^3 \colon  z=x   \right\}.
\end{aligned}
\]
Applying \Cref{image}, it is straightforward to verify that if the triplet of parameters $(\ell_0,  \ell_1,\rho)$ of
a Sturmian word $\vv $  belongs
\begin{enumerate}[(i)]
\item to $P_1$, then the  parameters of $\varphi_b(\vv)$  and the  parameters of $\varphi_\beta(\vv)$  belong to $P_1$;
\item to  $P_2$, then the  parameters of $\varphi_b(\vv)$  and the  parameters of $\varphi_\alpha(\vv)$   belong to $P_2$;
\item to  $P_3$, then the  parameters of $\varphi_a(\vv)$  and the  parameters of $\varphi_\beta(\vv)$   belong to $P_3$;
\item  to $P_4$, then the  parameters of $\varphi_a(\vv)$  and the  parameters of $\varphi_\alpha(\vv)$   belong to $P_4$.
\end{enumerate}

Let us start with \Cref{it:vRovinePulka4} and assume  that $\uu$ is fixed by a morphism $\varphi_w$ with  $w \in \{ a,\alpha\}^*$.
Since the word $\varphi_w(0)$ has a prefix $0$ and  $\varphi_w(1)$ has a prefix $1$, the morphism $ \varphi_w$  has two fixed points. Denote $\uu^{(1)}$  the lower Sturmian word  coding the two interval exchange with the domain  $[0, \ell_0+\ell_1)$ and the initial point $\rho =\ell_0 \in I_1=[\ell_0, \ell_0+\ell_1)$.
Let $\uu^{(0)}$  denote  the upper  Sturmian word  coding two interval exchange  with the domain  $(0, \ell_0+\ell_1]$ and the initial point $\rho =\ell_0 \in I_0$.
It means that  $\uu$, $\uu^{(0)}$  and $\uu^{(1)}$ have the same slope $\gamma = \tfrac{\ell_1}{\ell_0+\ell_1}$.
The word $\uu$ is fixed  by  $\varphi_w$ and thus the slopes $\varphi(\uu)$ and $\uu$ are the same, namely $\gamma$.
By \Cref{interceptzustava},   the slope of $\varphi_w(\uu^{(0)})$  and $\varphi_w(\uu^{(1)})$  is $\gamma$ as well.
Moreover, parameters   of  $\uu^{(0)}$  and $\uu^{(1)}$ belong to the plane $P_4$, which is preserved under the action of $\varphi_w$.
It follows that    $\uu^{(0)}$  and $\uu^{(1)}$  are the two fixed points of $\varphi_w$ and thus $\uu$ equals  $\uu^{(0)}$  or $\uu^{(1)}$.
Both these words have the initial point $\rho=\ell_0$.

\medskip

The proof of Items \ref{it:vRovinePulka1}--\ref{it:vRovinePulka3} is analogous. The only difference is that the morphism $\varphi_w$ has only one fixed point.

\end{proof}

\begin{lem}\label{vRovinePulka2} Let $\theta$ be a quadratic irrational such that $0<\theta <1$, $\overline{\theta}<0$  and let $\rho \in \mathbb{R}$, $0\leq \rho \leq 1+\theta$.
Let $\uu$  be a Sturmian word  with parameters $\ell_0 = 1, \ell_1 =\theta$ and $\rho$, or with parameters  $\ell_0=\theta, \ell_1= 1$ and $\rho$.
\begin{enumerate}[(1)]
\item  If  $\rho = \ell_1$, then  $\uu$ is fixed by a primitive morphism $\varphi_w$ with  $w \in \{ b,\beta\}^*$; \label{it:vRovine1}
\item  If  $\rho = 0$, then  $\uu$ is fixed by a primitive morphism $\varphi_w$ with   $w \in \{ b,\alpha\}^*$; \label{it:vRovine2}
\item  If   $\rho =\ell_0+\ell_1 $,  then  $\uu$ is fixed by a primitive morphism $\varphi_w$ with  $w \in \{ a,\beta\}^*$; \label{it:vRovine3}
\item  If  $\rho = \ell_0$, then  $\uu$ is fixed by a primitive morphism $\varphi_w$ with    $w \in \{ a,\alpha\}^*$. \label{it:vRovine4}
 \end{enumerate}
\end{lem}
\begin{proof}
We assume without loss of generality that $\uu$ has parameters $\ell_0 = 1, \ell_1 =\theta$ and $\rho$. In particular, the slope of $\uu$ is $\gamma = \tfrac{\theta}{1+\theta}$.
  The parameter $\theta$ and all four possible choices  for the parameter $\rho$ from the set  $\{\ell_1,0,\ell_0+\ell_1, \ell_0\} = \{ 1, \theta,  0, \theta+1\}$ satisfy the  Yasutomi condition in \Cref{Yasutomi} and thus  in all four cases $\uu$ is fixed by a primitive morphism.

\Cref{it:vRovine1}:  If $\rho = \ell_1=\theta $, then $\uu$ is a standard Sturmian word.   By  \cite{Crisp},  it is fixed by a standard substitution, that is, by a substitution $\varphi_u$ with  $u= u_1u_2\cdots u_n \in \{b, \beta\}^*$.

\Cref{it:vRovine2}:
We use the word $u= u_1u_2\cdots u_n \in \{b, \beta\}^*$ from the proof of \Cref{it:vRovine1} to define the word $v= v_1v_2\cdots v_n \in \{b, \alpha\}^*$.
We set $v_k= b$ if $u_k= b$, and $v_k= \alpha$ if $u_k= \beta$.
Let $\vv$ denote the fixed point of $\varphi_v$
 By \Cref{interceptzustava},  the  fixed point of $\varphi_v$ has the same slope as the fixed point of $\varphi_u$, namely  $\gamma = \tfrac{\theta}{1+\theta}$.
 By  \Cref{vRovinePulka1},  the third parameter of $\vv$ is $0$.
 It means that $\vv$ has parameters  $(\ell_0, \ell_1, 0)$  and coincides with $\uu$.
 Consequently, the word $\uu$ is fixed by $\varphi_v$ with $v\in \{b, \alpha\}^*$.

The proof of the remaining two parts is analogous.
\end{proof}

\begin{lem} \label{neniFixed}
Let $\vv$ be a Sturmian word fixed by a primitive substitution $\psi = \varphi_{w}$, where $w =\beta e b a^k$ for some $e\in \{a,b,\alpha,\beta\}^*$ and $k\geq 0$.
No primitive substitution  fixes the Sturmian word  $\sigma(\vv)$.
\end{lem}

\begin{proof}
As $w$ starts with $\beta$,  the letter $1$ is the more frequent letter in $\vv$ and $1$ is the starting letter of $\vv$.
Thus, $\vv$ is coding of the interval exchange with parameters  $\ell_0=\theta \in (0,1)$, $\ell_1 = 1$ and  $\rho$ satisfying  $\theta \leq \rho \leq 1+\theta$.
The word $\vv$  can be written in the form $1\vv'$, where $\vv' = \sigma(\vv)$.
By \Cref{shift}, the word $\vv'$ has parameters $\ell_0'=\theta$, $\ell_1' = 1$ and  $\rho'=\rho-\theta$.

We prove by contradiction that no substitution fixes $\vv'$.
Let us assume that $\vv'$ is fixed by a primitive substitution.
\Cref{Yasutomi} gives
\begin{equation} \label{nerovnosti}
\overline{\theta}\leq \overline{ \rho} - \overline{\theta}\leq 1\quad \text{and} \quad \overline{\theta} < 0.
 \end{equation}
Put  $w^{(1)} = \beta e$ and $w^{(2)} = ba^k$ and denote $\psi_1=\varphi_{w^{(1)}}$ and $\psi_2=\varphi_{w^{(2)}}$.
Clearly, $\psi = \psi_1\circ\psi_2$ and
\[
\vv = \psi_1\left(\psi_2(\vv)\right) \ \ \Longrightarrow \ \  \psi_2(\vv) = \psi_2\psi_1\left(\psi_2(\vv)\right).
\]
Thus the word  $\vv''=\psi_2(\vv)$ is fixed by the  primitive substitution $\psi_2\circ\psi_1$.
By \Cref{image}, the word $\vv'' $  has parameters $ 1+k +\theta, 1,  k+\rho$, which we normalize to
$\ell_0'' =1, \ell_1'' = \tfrac{1}{1+k +\theta}$ and $ \rho'' =\tfrac{ k+\rho}{1+k +\theta}$.
These parameters  satisfy the condition given by \Cref{Yasutomi}, i.e.,
\begin{equation}\label{nerovnosti2}
\frac{1}{1+k +\overline{\theta}}<0
\quad \text{ and } \quad  \frac{1}{1+k +\overline{\theta}}  \leq  \frac{ k+\overline{\rho}}{1+k +\overline{\theta}} \leq  1.
 \end{equation}
By  \eqref{nerovnosti2}   and  \eqref{nerovnosti}  we obtain $\overline{\rho} \leq \overline{\theta} +1$ and $\overline{\rho} \geq \overline{\theta} +1$, respectively.
It means that $\rho = 1+\theta = \ell_0+\ell_1$.
By \Cref{it:vRovine3} of \Cref{vRovinePulka2}, the word $\vv$ is fixed by a substitution $\varphi_u$ with $u \in \{\beta, a\}^*$.
Since such word $u$ cannot be rewritten using \eqref{eq:relations}, by \Cref{thm:relations} the substitution $\varphi_u$ and every its power are elements of $\langle \varphi_\beta, \varphi_a \rangle$.
Similarly by \Cref{thm:relations} we obtain $\varphi_w^k \not \in \langle \varphi_\beta, \varphi_a \rangle$ for all $k \in \N$.
Since every Sturmian word is rigid (see \cite{RiSe12}), we obtain a contradiction as the Sturmian word $\vv$ cannot be fixed by both $\varphi_u$ and $\varphi_w$.
\end{proof}

\begin{proof}[Proof of \Cref{neniReflex}]
Let  ${\bf w}$  be a fixed point of $\psi = \varphi_w$.
By \Cref{thm:main_result_vetsina}, the  derived word to any prefix of  ${\bf w}$  is fixed by a substitution $\Delta^i(w)$ for some $i$.
By \Cref{vsechny4}  there exists $i_0\in \mathbb{N}$ such  that  for each $i>i_0$,  both letters $b$ and $\beta$ occur in $\Delta^i(w)$ and $\Delta(\Delta^i(w)) = {\rm cyc}_j(\Delta^i(w))$ for some $j$. By the same remark,  at least one of the letters from $\{a,\alpha\}$ occurs in $\Delta^i(w)$.
Therefore, there exists $i>i_0$ such that
$$\text{for some } \  e \in\{a,b,\alpha,\beta\}^*   \  \text{and} \  k\in \mathbb{N}, k\geq 1 ,\   \  \text{ either  }
\Delta^i(w) = a^k\beta e b \ \text{ or } \ \Delta^i(w) = \alpha^kb e \beta.$$
We first treat the case $ \Delta^i(w) = a^k\beta e b $.
Let $\uu$ denote the fixed point of  $\varphi_{\Delta^i(w)}$.
It follows that $ \uu$ has an S-adic representation $(a^k\beta e b)^\omega$.

Let $\vv$ denote  the Sturmian word with the S-adic representation $(\beta e ba^k)^\omega$.
The word $\vv$ begins  with the letter $1$ and thus  $\vv=1 \vv'$ for some Sturmian word $\vv'$.
Moreover,
\[
\uu= \varphi_a^k(\vv) =  \varphi_a^k(1\vv').
\]
By  Corollary \ref{jinak} \Cref{Claim2}, $\uu = 1\varphi_b^k(\vv')$.
By Corollary \ref{jinak}   \Cref{Claim1},  $0^k$ is a factor of $\uu$ and $\dd_\uu(0^k) = \vv' = \sigma(\vv)$.
\Cref{neniFixed} implies that $\vv'$ is not fixed by any primitive substitution.

\medskip

To sum up:
\begin{itemize}
\item $\uu$ is  a derived word  to a prefix of the fixed point  $\ww$  of the substitution $\psi$;
\item $0^k$ is a factor of $\uu$ and the derived word of $\dd_\uu(0^k) $ is not fixed by any primitive substitution.
\end{itemize}

It implies that the fixed point $\ww$ of the primitive substitution  $\psi$ is not closeable under derivation.

The second case $\Delta^i(w) = \alpha^kb e \beta$ follows easily from the first case by exchanging the letters $0$ and $1$.
\end{proof}

\section{closeable under derivation Sturmian substitutions} \label{sec:deriv}

The aim of this section  is to prove the following theorem.
\begin{thm} \label{jeReflex}
If $\psi= \varphi_w$ is a primitive Sturmian substitution such that  $w \in \{b,
\beta\}^*\cup\{b,
\alpha\}^*\cup\{a,
\beta\}^*\cup\{a,
\alpha\}^*$,
then $\psi$ is closeable under derivation.
\end{thm}

The proof will be split into four cases according to the couple of letters which appear in the name $w$ determining the substitution $\varphi_w$.
To abbreviate the notation we set
\[
C(w) = \{\varphi_{v} \colon  v={\rm cyc}^k(w), k \in \mathbb{N}  \} \quad \text{ for } w \in \{a,b,\alpha, \beta\}^*.
\]

\subsection{Case \texorpdfstring{$w \in \{b,\beta\}^*$}{w in {b,beta}*}} \label{standard}

The substitution $\varphi_w$  is composed from the  standard  Sturmian substitutions  $\varphi_\beta$ and $\varphi_b$.
Therefore, the fixed point $\uu$ of  $\varphi_w$ is a standard Sturmian word.
Any  bispecial factor of a standard Sturmian word is one of its prefixes.
Thus by \Cref{prolong} the derived word to an arbitrary factor of $\uu$ coincides with the derived word to a prefix of $\uu$.
\Cref{thm:main_result_vetsina} implies that  the derived word to a prefix of $\uu$ is fixed by a substitution $\varphi_v$ with $v={\rm cyc}^k(w)$ for some $k$, i.e., $\varphi_v \in C(w)$.
Clearly, $v$ belongs to $\{b,\beta\}^*$ and thus $\varphi_v$ is again a standard Sturmian substitution.
Thus, we may repeat this argument, and we can conclude the following.

\begin{claim}\label{betab}
For any  $w \in \{b,\beta\}^*$,  the set $C(w)$ is closed under derivation.
\end{claim}

Discussion of the other  cases uses a consequence of \Cref{vRovinePulka1,vRovinePulka2}.
Recall that the shift operator $\sigma$ erases the starting letter of an infinite word, i.e., maps the word $\uu = u_0u_1u_2 \dots$ to the word $\sigma(\uu) = u_1u_2u_3 \dots$.
To abbreviate the notation we define two projections   $H$ and $F$: $\{a,b,\alpha,\beta\}^* \to \{a,b, \alpha,\beta\}^*$ by
\[
H(a)=H(b)=b, H(\alpha)= \alpha,  H(\beta) = \beta \quad \text{ and } \quad F(a)=a, F(b)=b, F(\alpha)= F(\beta) = \beta.
\]

\begin{lem}\label{Shift}
Let $ \uu$ be fixed by a substitution $\varphi_w$.
\begin{enumerate}[(1)]
\item If $w \in \{ a,\beta\}^*$, then  $\sigma(\uu)$ is a standard Sturmian word which is fixed by the substitution $\varphi_{H(w)}$. \label{it:Shift1}
\item If  $w \in \{ b,\alpha\}^*$, then  $\sigma(\uu)$ is a standard Sturmian word which is fixed by the substitution $\varphi_{F(w)}$. \label{it:Shift2}
\item If $w \in \{ a,\alpha\}^*$ and $\uu$ has a prefix $1$, then  $\sigma(\uu)$ is fixed by the substitution $\varphi_{H(w)}$. \label{it:Shift3}
\item  If $w \in \{ a,\alpha\}^*$ and $\uu$ has a prefix $0$,  then  $\sigma(\uu)$ is fixed by the substitution $\varphi_{F(w)}$.  \label{it:Shift4}
\end{enumerate}
\end{lem}

\begin{proof}
 Let $(\ell_0, \ell_1, \rho)$ be the parameters  of  $\uu$.

\Cref{it:Shift1}: By \Cref{vRovinePulka1}, $\uu$  has parameters $\ell_0, \ell_1, \rho= \ell_0 +\ell_1$.
In particular, $\uu$  is an  upper Sturmian word.
By \Cref{shift}, $\sigma(\uu)$  has parameters $\ell_0, \ell_1, \rho= \ell_1$, i.e.,  $\sigma(\uu)$  is a standard Sturmian word.
Clearly,  the slopes of $\sigma(\uu)$ and $\uu$  coincide.
Using \Cref{interceptzustava,vRovinePulka2}, the word $\sigma(\uu)$ is fixed by the  substitution  $\varphi_{H(w)}$.

\Cref{it:Shift2}: Analogous to the proof of \Cref{it:Shift1}.

\Cref{it:Shift3}:    The substitution $\varphi_w$ has two fixed points, one starting with the letter  $0$ and one starting with  the letter $1$.
By \Cref{vRovinePulka1}, both  fixed points represent a coding of a two interval exchange transformation $T$ with  parameter $\rho$  satisfying   $\rho = \ell_0$.
If $\uu$  starts with $1$,  then  $\uu$  is a coding of the transformation $T$ with  the domain $[0, \ell_0+\ell_1)$.
In this case,   $T(\rho) = 0$ and  by \Cref{vRovinePulka2} the word $ \sigma(\uu)$  is  fixed by the substitution $\varphi_{H(w)}$.

\Cref{it:Shift4}: Analogous to the proof of \Cref{it:Shift3}.
\end{proof}

\subsection{ Case \texorpdfstring{$w\in  \{b,\alpha\}^*$}{w in {b,alpha}*}}

Theorem \ref{thm:main_result_vetsina} implies that a derived word to a prefix of $\uu$  is fixed by one of the substitutions  from $C(w)$.
 Using \Cref{prolong}, we  can focus on  derived words to bispecial non-prefixes of $\uu$.
 By \Cref{Shift}, the word $\sigma(\uu)$ is a standard Sturmian word and thus any bispecial factor $v$ of $\uu$  occurs as a prefix of $\sigma(\uu)$.
 Therefore,   $\dd_{\uu}(v) =\dd_{\sigma(\uu)}(v)$.
 By \Cref{Shift},  the word $\sigma(\uu)$ is fixed by the standard substitution $\varphi_{F(w)}$, i.e.,  $F(w) \in \{b, \beta\}^*$.
 Using \Cref{betab} we conclude the following.

\begin{claim}\label{alphab}
For any  $w \in \{b, \alpha\}^*$,  the set $C(w) \cup C\bigl(F(w)\bigr)$ is closed under derivation.
\end{claim}

\subsection{Case  \texorpdfstring{$w\in \{a, \beta\}^*$}{w in {a,beta}*}}

This case is analogous to the previous one. Indeed,
the words $\uu$ and $E(\uu)$ have the same  (up to the permutation of letters) set of derived words.
Since  $\uu$ is fixed  by  $\varphi_w$, the word   $E(\uu)$ is fixed by $E\varphi_w E$.
As $\varphi_a = E\varphi_\alpha E$ and  $\varphi_b = E\varphi_\beta E$, the substitutions $E\varphi_w E = \varphi_{v}$, where $v \in   \{b,
\alpha\}^*$.

\begin{claim}\label{abeta}
For any  $w \in \{a, \beta\}^*$,  the set $C(w) \cup C\bigl(H(w)\bigr)$ is closed under derivation.
\end{claim}

\subsection{ Case \texorpdfstring{$w\in \{a,\alpha\}^*$}{w in {a,alpha}*}}

The substitution $\varphi_w$ has two fixed points.
The derived words to  prefixes of the fixed points of  $\varphi_w$ are described in \cite{KlMePeSt18}:
\begin{prop}[\cite{KlMePeSt18}] \label{lem:revers_of_standard}
	Let $a$ be the first letter of the word $w \in\{a,\alpha\}^*$.
	\begin{enumerate}[(i)]
		\item \label{it:reverse:1} Let $\uu$ be the fixed point of $\varphi_w$ starting with $0$ and $p$ be a non-empty prefix of $\uu$.
		Denote $v =b^{-1}N(wb)\in \{a,\beta\}^*$.
		The derived word  ${\bf d_{\bf u}}(p)$  equals a derived word ${\bf d_{\bf v}}(q)$, where
		 ${\bf v }$ is the unique fixed point of the substitution $\varphi_v$ and $q$ is a prefix of $\vv$.

		\item \label{it:reverse:2} Let $\uu$ be the fixed point of $\varphi_w$ starting with $1$ and $p$ be a non-empty prefix of $\uu$.
		Put $v ={\rm cyc} (w)$.
		The word ${\bf d_{\bf u}}(p)$  equals a derived word ${\bf d_{\bf v}}(q)$, where
		 ${\bf v }$ is the  fixed point of the substitution $\varphi_v$ starting with $1$  and $q$ is a prefix of $\vv$.
	\end{enumerate}
\end{prop}
An analogous proposition can be stated if $\alpha$ is the first letter of $w \in\{a,\alpha\}^*$.
In this case, the roles of the letters $0$ and $1$ in \Cref{it:reverse:1,it:reverse:2} are then  interchanged.
In particular, the word $v$  in  \Cref{it:reverse:1} is defined by $v = \beta^{-1}N(w\beta)\in \{b,\alpha\}^*$.
Nevertheless, in the following example we show that in both cases the word $v$ belongs either to $\{ {\rm cyc}^k(F(w)) : k  \in \mathbb{N}\}$ or to $
\{ {\rm cyc}^k(H(w)) : k  \in \mathbb{N}\} $.



\begin{example}
Consider $w=a^4\alpha^2a^2\alpha a^3 \in \{a,\alpha\}^*$.
The starting letter of this word is $a$.
The word $v =b^{-1}N(wb)$  from   \Cref{lem:revers_of_standard} satisfies $v= a^3\beta^2a^2\beta a^4$.
Therefore, we have $v = {\rm cyc}(F(w))$.

Consider $w= \alpha a^3\alpha^4a \in \{a,\alpha\}^* $.
It follows that $v=\beta^{-1}N(w\beta) = b^3\alpha^4b$, and thus $v = {\rm cyc}(H(w))$.
\end{example}

\Cref{lem:revers_of_standard}  has the following direct corollary.

\begin{coro}\label{proprefix}
Let $\uu$ be a fixed point of $\varphi_w$  with $w \in \{a,\alpha\}^*$  and $p\neq \varepsilon$ be a prefix of $\uu$.
The word ${\bf d_{\bf u}}(p)$  is fixed by a substitution from $C(F(w)) \cup C(H(w))$.
\end{coro}

\begin{claim}\label{alphaa}  For any  $w \in \{a,
\alpha\}^*$,  the set $C(w) \cup C\bigl(H(w)\bigr)\cup C\bigl(F(w)\bigr) \cup C\bigl(HF(w)\bigr) $ is closed under derivation.
\end{claim}

\begin{proof}
We deduce  a stronger statement, namely that the  set  $M:= \{\varphi_w\}\cup C\bigl(H(w)\bigr)\cup C\bigl(F(w)\bigr) \cup C\bigl(HF(w)\bigr) $ is  closed under derivation.

First, we realize that  $H(w) \in \{b,
\alpha\}^*$, $F(w) \in \{a,
\beta\}^*$ and $HF(w) \in\{b,
\beta\}^*$.
By virtue of  \Cref{betab,alphab,abeta},  the set  $N:=M\bigl(H(w)\bigr)\cup C\bigl(F(w)\bigr) \cup C\bigl(HF(w)\bigr) $ is closed under derivation.
To demonstrate that $M= \{\varphi_w\} \cup N$ is closed under derivation,  we  only need to show that the
derived word to any factor $v$  of any fixed point $\uu$ of $\varphi_w$ is fixed by a morphism from  $M$.
Indeed:
\begin{enumerate}[a)]
\item if   $v = \varepsilon$,   then  ${\bf d_{\bf u}}(\varepsilon) = \uu$ and thus ${\bf d_{\bf u}}(\varepsilon)$ is fixed by $\varphi_w \in M$;
\item if   $v $ is a non-empty prefix of $\uu$,  then  by \Cref{proprefix}  ${\bf d_{\bf u}}(v) $ is fixed by a substitution from  $N\subset M$;
\item if   $v $ is a  non-prefix factor  of $\uu$,  then $v$ is a factor of $\sigma(\uu)$  and      $\dd_{\uu}(v) =\dd_{\sigma(\uu)}(v)$. By \Cref{it:Shift3,it:Shift4} of \Cref{Shift}, the word  $\sigma(\uu)$ is fixed by a substitution from $N$. Since $N$ is closed under  derivation,  ${\bf d_{\bf u}}(v) $ is fixed by a substitution from  $N\subset M$. \qedhere 	
\end{enumerate}
\end{proof}

The following theorem is a direct consequence of \Cref{vRovinePulka1,neniFixed} and \Cref{neniReflex}.
\begin{theorem} \label{thm:last}
Let $\uu$ code the two interval exchange transformation with parameters $\ell_0 =1- \gamma, \ell_1 = \gamma, \rho = \delta$,  where $\gamma, \delta \in [0,1]$ and $\gamma$ irrational.
If $\uu$ is fixed by a primitive substitution $\varphi$, then $\varphi$ is closeable under derivation if and only if $\delta \in \{0,\gamma, 1-\gamma, 1\}$.
\end{theorem}

\section{Comments}

We presented examples of finite sets $M$ of Sturmian substitutions that are closed under derivation.
We used two tools: $S$-adic representation of Sturmian words and algebraic characterization of Sturmian words that are fixed by a primitive substitution.
Probably the most explored class of words generalizing Sturmian words is the class of ternary Arnoux--Rauzy words.
Therefore, it is natural to generalize our results to this class, more specifically, to the substitutions fixing ternary Arnoux--Rauzy words.
However, there is no analogue of the Yasutomi's characterization of fixed points in this class.
Another well studied class that generalizes Sturmian words are words coding $k$-interval exchange transformations.
In the case $k=3$ and the permutation of interval exchange being $(321)$, an analogue to Yasutomi's conditions is provided in \cite{BaMaPe2008}.



Our example of sets $M$ which are closed under derivation are composed of Sturmian substitutions and thus all elements of $M$ act on the same alphabet.
The same property would hold for substitutions fixing ($k$-ary) Arnoux--Rauzy words and for substitutions fixing three interval exchange with the permutation $(321)$.
However, there exists an example of a set $M$ closed under derivation which contains substitutions acting on a binary alphabet and substitutions acting on a ternary alphabet such that no proper subset of $M$ is closed under derivation.
This example is given in~\cite{KoSta19} and the set $M$ contains substitutions fixing derived words to non-empty factors of the period doubling sequence.
Recall that the period doubling sequence is fixed by the substitution  $a \mapsto ab, b\mapsto aa$.
The derived words of the period doubling sequence were described in~\cite{HuangWen17}.

\section*{Acknowledgements}
Edita Pelantová acknowledges financial support by the Ministry of Education, Youth and Sports of the Czech Republic, project no. CZ.02.1.01/0.0/0.0/16\_019/0000778.
Štěpán Starosta acknowledges the support of the OP VVV MEYS funded project
CZ.02.1.01/0.0/0.0/16\_019/0000765. 
The authors are also grateful for the hospitality of Erwin Schr\"odinger International Institute for Mathematics and Physics, where a part of the work was done.

\bibliographystyle{siam}

\end{document}